\documentclass[a4paper,10pt]{article}
\usepackage[reqno]{amsmath}
\makeatletter
\newcommand{\leqnomode}{\tagsleft@true}
\newcommand{\reqnomode}{\tagsleft@false}
\makeatother

\usepackage[english,british]{babel}
\usepackage{amsmath,amsthm,amssymb,dsfont,CJK,esint}
\usepackage{adjustbox,lipsum}
\usepackage{tikz}
\usepackage{standalone}
\usepackage{mathtools}
\usepackage{fancyhdr}
\usepackage{graphicx}
\usepackage{bibentry}
\usepackage{enumitem}
\usepackage{xcolor}
\usepackage{stackengine}
\usepackage{hyperref}
\usepackage[a4paper,width=160mm,top=30mm,bottom=30mm]{geometry}

\newtheorem{theorem}{Theorem}[section]
\newtheorem{lemma}[theorem]{Lemma}

\theoremstyle{definition}
\newtheorem{remx}[theorem]{Remark}
\newenvironment{remark}
  {\pushQED{\qed}\remx}
  {\popQED\endremx}

\makeatletter
\newsavebox\myboxA
\newsavebox\myboxB
\newlength\mylenA

\newcommand*\yoverline[2][0.75]{%
    \sbox{\myboxA}{$\m@th#2$}%
    \setbox\myboxB\null
    \ht\myboxB=\ht\myboxA%
    \dp\myboxB=\dp\myboxA%
    \wd\myboxB=#1\wd\myboxA
    \sbox\myboxB{$\m@th\overline{\copy\myboxB}$}
    \setlength\mylenA{\the\wd\myboxA}
    \addtolength\mylenA{-\the\wd\myboxB}%
    \ifdim\wd\myboxB<\wd\myboxA%
       \rlap{\hskip 0.5\mylenA\usebox\myboxB}{\usebox\myboxA}%
    \else
        \hskip -0.5\mylenA\rlap{\usebox\myboxA}{\hskip 0.5\mylenA\usebox\myboxB}%
    \fi}
\makeatother

\numberwithin{equation}{section}

\begin{document}




\newcommand{\diver}{\operatorname{div}}
\newcommand{\lin}{\operatorname{Lin}}
\newcommand{\curl}{\operatorname{curl}}
\newcommand{\ran}{\operatorname{Ran}}
\newcommand{\kernel}{\operatorname{Ker}}
\newcommand{\la}{\langle}
\newcommand{\ra}{\rangle}
\newcommand{\N}{\mathds{N}}
\newcommand{\R}{\mathds{R}}
\newcommand{\C}{\mathds{C}}

\newcommand{\ld}{\lambda}
\newcommand{\fai}{\varphi}
\newcommand{\0}{0}
\newcommand{\n}{\mathbf{n}}
\newcommand{\uu}{{\boldsymbol{\mathrm{u}}}}
\newcommand{\UU}{{\boldsymbol{\mathrm{U}}}}
\newcommand{\buu}{\bar{{\boldsymbol{\mathrm{u}}}}}
\newcommand{\ten}{\\[4pt]}
\newcommand{\six}{\\[4pt]}
\newcommand{\nb}{\nonumber}
\newcommand{\hgamma}{H_{\Gamma}^1(\OO)}
\newcommand{\opert}{O_{\varepsilon,h}}
\newcommand{\barx}{\bar{x}}
\newcommand{\barf}{\bar{f}}
\newcommand{\hatf}{\hat{f}}
\newcommand{\xoneeps}{x_1^{\varepsilon}}
\newcommand{\xh}{x_h}
\newcommand{\scaled}{\nabla_{1,h}}
\newcommand{\scaledb}{\widehat{\nabla}_{1,\gamma}}
\newcommand{\vare}{\varepsilon}
\newcommand{\A}{{\bf{A}}}
\newcommand{\RR}{{\bf{R}}}
\newcommand{\B}{{\bf{B}}}
\newcommand{\CC}{{\bf{C}}}
\newcommand{\D}{{\bf{D}}}
\newcommand{\K}{{\bf{K}}}
\newcommand{\oo}{{\bf{o}}}
\newcommand{\id}{{\bf{Id}}}
\newcommand{\E}{\mathcal{E}}
\newcommand{\ii}{\mathcal{I}}
\newcommand{\sym}{\mathrm{sym}}
\newcommand{\lt}{\left}
\newcommand{\rt}{\right}
\newcommand{\ro}{{\bf{r}}}
\newcommand{\so}{{\bf{s}}}
\newcommand{\e}{{\bf{e}}}
\newcommand{\ww}{{\boldsymbol{\mathrm{w}}}}
\newcommand{\vv}{{\boldsymbol{\mathrm{v}}}}
\newcommand{\zz}{{\boldsymbol{\mathrm{z}}}}
\newcommand{\U}{{\boldsymbol{\mathrm{U}}}}
\newcommand{\G}{{\boldsymbol{\mathrm{G}}}}
\newcommand{\VV}{{\boldsymbol{\mathrm{V}}}}
\newcommand{\T}{{\boldsymbol{\mathrm{U}}}}
\newcommand{\II}{{\boldsymbol{\mathrm{I}}}}
\newcommand{\ZZ}{{\boldsymbol{\mathrm{Z}}}}
\newcommand{\hKK}{{{\bf{K}}}}
\newcommand{\f}{{\bf{f}}}
\newcommand{\g}{{\bf{g}}}
\newcommand{\lkk}{{\bf{k}}}
\newcommand{\tkk}{{\tilde{\bf{k}}}}
\newcommand{\W}{{\boldsymbol{\mathrm{W}}}}
\newcommand{\Y}{{\boldsymbol{\mathrm{Y}}}}
\newcommand{\EE}{{\boldsymbol{\mathrm{E}}}}
\newcommand{\F}{{\bf{F}}}
\newcommand{\spacev}{\mathcal{V}}
\newcommand{\spacevg}{\mathcal{V}^{\gamma}(\Omega\times S)}
\newcommand{\spacevb}{\bar{\mathcal{V}}^{\gamma}(\Omega\times S)}
\newcommand{\spaces}{\mathcal{S}}
\newcommand{\spacesg}{\mathcal{S}^{\gamma}(\Omega\times S)}
\newcommand{\spacesb}{\bar{\mathcal{S}}^{\gamma}(\Omega\times S)}
\newcommand{\skews}{H^1_{\barx,\mathrm{skew}}}
\newcommand{\kk}{\mathcal{K}}
\newcommand{\OO}{O}
\newcommand{\bhe}{{\bf{B}}_{\vare,h}}
\newcommand{\pp}{{\mathbb{P}}}
\newcommand{\ff}{{\mathcal{F}}}
\newcommand{\mWk}{{\mathcal{W}}^{k,2}(\Omega)}
\newcommand{\mWa}{{\mathcal{W}}^{1,2}(\Omega)}
\newcommand{\mWb}{{\mathcal{W}}^{2,2}(\Omega)}
\newcommand{\twos}{\xrightharpoonup{2}}
\newcommand{\twoss}{\xrightarrow{2}}
\newcommand{\bw}{\bar{w}}
\newcommand{\br}{\bar{{\bf{r}}}}
\newcommand{\bz}{\bar{{\bf{z}}}}
\newcommand{\tw}{{W}}
\newcommand{\tr}{{{\bf{R}}}}
\newcommand{\tz}{{{\bf{Z}}}}
\newcommand{\lo}{{{\bf{o}}}}
\newcommand{\hoo}{H^1_{00}(0,L)}
\newcommand{\ho}{H^1_{0}(0,L)}
\newcommand{\hotwo}{H^1_{0}(0,L;\R^2)}
\newcommand{\hooo}{H^1_{00}(0,L;\R^2)}
\newcommand{\hhooo}{H^1_{00}(0,1;\R^2)}
\newcommand{\dsp}{d_{S}^{\bot}(\barx)}
\newcommand{\LB}{{\bf{\Lambda}}}
\newcommand{\LL}{\mathbb{L}}
\newcommand{\mL}{\mathcal{L}}
\newcommand{\mhL}{\widehat{\mathcal{L}}}
\newcommand{\loc}{\mathrm{loc}}
\newcommand{\tqq}{\mathcal{Q}^{*}}
\newcommand{\tii}{\mathcal{I}^{*}}
\newcommand{\Mts}{\mathbb{M}}
\newcommand{\pot}{\mathrm{pot}}
\newcommand{\tU}{{\widehat{\bf{U}}}}
\newcommand{\tVV}{{\widehat{\bf{V}}}}
\newcommand{\pt}{\partial}
\newcommand{\bg}{\Big}
\newcommand{\hA}{\widehat{{\bf{A}}}}
\newcommand{\hB}{\widehat{{\bf{B}}}}
\newcommand{\hCC}{\widehat{{\bf{C}}}}
\newcommand{\hD}{\widehat{{\bf{D}}}}
\newcommand{\fder}{\partial^{\mathrm{MD}}}
\newcommand{\Var}{\mathrm{Var}}
\newcommand{\pta}{\partial^{0\bot}}
\newcommand{\ptaj}{(\partial^{0\bot})^*}
\newcommand{\ptb}{\partial^{1\bot}}
\newcommand{\ptbj}{(\partial^{1\bot})^*}
\newcommand{\geg}{\Lambda_\vare}
\newcommand{\tpta}{\tilde{\partial}^{0\bot}}
\newcommand{\tptb}{\tilde{\partial}^{1\bot}}
\newcommand{\ua}{u_\alpha}
\newcommand{\pa}{p\alpha}
\newcommand{\qa}{q(1-\alpha)}
\newcommand{\Qa}{Q_\alpha}
\newcommand{\Qb}{Q_\eta}
\newcommand{\ga}{\gamma_\alpha}
\newcommand{\gb}{\gamma_\eta}
\newcommand{\ta}{\theta_\alpha}
\newcommand{\tb}{\theta_\eta}


\newcommand{\mH}{\mathcal{H}}
\newcommand{\csob}{\mathcal{S}}
\newcommand{\mA}{\mathcal{A}}
\newcommand{\mK}{\mathcal{K}}
\newcommand{\mS}{\mathcal{S}}
\newcommand{\mI}{\mathcal{I}}
\newcommand{\tas}{{2_*}}
\newcommand{\tbs}{{2^*}}
\newcommand{\tm}{{\tilde{m}}}
\newcommand{\tdu}{{\tilde{u}}}
\newcommand{\tpsi}{{\tilde{\psi}}}

\selectlanguage{english}
\title{Scattering threshold for radial defocusing-focusing mass-energy double critical nonlinear Schr\"odinger equation in $d\geq 5$  }
\author{Yongming Luo \thanks{Institut f\"{u}r Wissenschaftliches Rechnen, Technische Universit\"at Dresden, Germany} \thanks{\href{mailto:yongming.luo@tu-dresden.de}{Email: yongming.luo@tu-dresden.de}}
}

\date{}
\maketitle

\begin{abstract}
We extend the scattering result given by Cheng et al. for the radial defocusing-focusing mass-energy double critical nonlinear Schr\"odinger equation in $d\leq 4$ to the whole range $d\geq 3$. The main ingredient is a suitable long time perturbation theory which is applicable for $d\geq 5$.
\end{abstract}

\section{Introduction and main results}
In this paper, we consider the defocusing-focusing mass-energy double critical nonlinear Schr\"odinger equation (DFDCNLS)
\begin{equation}\label{NLS}
i\pt_t u+\Delta u-|u|^{2_*-2}u+|u|^{2^*-2}u=0\quad\text{in $\R\times \R^d$}
\end{equation}
with $d\geq 5$, $\tas=2+\frac{4}{d}$ and $\tbs=2+\frac{4}{d-2}$. \eqref{NLS} is a special case of the NLS with combined nonlinearities
\begin{equation}\label{NLS0}
i\pt_t u+\Delta u+\mu_1|u|^{p_1-2}u+\mu_2|u|^{p_2-2}u=0\quad\text{in $\R\times \R^d$}
\end{equation}
with $d\geq 1$, $\mu_1,\mu_2\in\R$ and $p_1,p_2\in(2,\infty)$. \eqref{NLS0} is a prototype model in many applications of quantum physics such as nonlinear optics and Bose-Einstein condensation. For example, in the study of Bose-Einstein condensation, the nonlinearities $|u|^2 u$, $|u|^3u$ and $|u|^4u$ model the two-body interaction, quantum fluctuation and three-body interaction respectively. The signs $\mu_i$ can be tuned to be defocusing ($\mu_i<0$) or focusing ($\mu_i>0$), indicating the repulsivity or attractivity of the nonlinearity. For a comprehensive introduction on the physical background of \eqref{NLS0}, we refer to \cite{phy_double_crit_1,phy_double_crit_2} and the references therein.

From a mathematical point of view, we are particularly interested in problems with critical nonlinearities due to the following aspects: On the one hand, the nonlinear estimates for non-critical problems can usually be derived from the critical ones by means of interpolation; on the other hand, by dealing with critical problems additional symmetry operator such as dilation or Galilean boosts will also come into play, which makes the problem more challenging and interesting. The above mentioned reasons hence motivate our study on the mass-energy double critical NLS, whose mixed type nature also prevents any potential applications concerning scaling invariance property. At this point, we also refer the readers to the representative papers \cite{TaoVisanZhang,Akahori2013,MiaoDoubleCrit,killip_visan_soliton,Cheng2020,Carles_Sparber_2021,luo2021sharp} for scattering results of \eqref{NLS0}, in which at least one of the nonlinearities has critical growth.

We restrict our attention to the radial DFDCNLS \eqref{NLS}, which was studied by Cheng, Miao, Zhao \cite{MiaoDoubleCrit} in the case $d\leq 4$. The precise statement is as follows:
\begin{theorem}[\cite{MiaoDoubleCrit}]\label{thm miao}
Let $d\in\{3,4\}$. Define
\begin{align*}
\mH(u)&:=\frac{1}{2}\|\nabla u\|_2^2+\frac{1}{\tas}\|u\|_\tas^\tas-\frac{1}{\tbs}\|u\|_\tbs^\tbs,\nonumber\\
\mK(u)&:=\|\nabla u\|_2^2+\frac{d}{d+2}\|u\|_\tas^\tas-\|u\|_\tbs^\tbs
\end{align*}
and
\begin{align*}
\mA:=\{u\in H_{\mathrm{rad}}^1(\R^d):\mH(u)<d^{-1}\mS^{\frac{d}{2}},\mK(u)\geq 0\},
\end{align*}
where $\mS$ is the optimal constant of the Sobolev inequality, i.e.
\begin{align*}
\mS:=\inf_{\mathcal{D}^{1,2}(\R^d)\setminus\{0\}}\frac{\|\nabla u\|_2^2}{\|u\|_\tbs^2}.
\end{align*}
Then the unique solution $u$ of \eqref{NLS} with $u(0)\in\mA$ is global and scatters in time.
\end{theorem}
The main obstacle that prevents Theorem \ref{thm miao} to hold in $d\geq 5$ is the absence of a suitable long time perturbation theory. More precisely, since the gradient of the nonlinearity $|u|^{\frac{4}{d}}u$ is merely H\"older continuous for $d\geq 5$, the proof of the long time perturbation result for $d\leq 4$ is no longer valid. By appealing to fractional calculus we show that such a long time perturbation result  indeed continues to hold also in the case $d\geq 5$.
\begin{theorem}\label{long time pert}
Let $d\geq 5$ and let $u\in C(I;H^1(\R^d))$ be a solution of \eqref{NLS} defined on some interval $I\ni t_0$. Assume also that $w$ is an approximate solution of the following perturbed NLS
\begin{align}
i\pt_t w+\Delta w=|w|^{\frac{4}{d}}w-|w|^{\frac{4}{d-2}}w+e
\end{align}
such that
\begin{align}
\|u\|_{L_t^\infty H_x^1(I)}&\leq B_1,\label{condition c1}\\
\|u(t_0)-w(t_0)\|_{H^1}&\leq B_2,\label{condition c3}\\
\|w\|_{W_\tas\cap W_\tbs (I)}&\leq B_3\label{condition c2}
\end{align}
for some $B_1,B_2,B_3>0$. Then there exists some positive $\alpha=\alpha(B_1,B_2,B_3)\ll 1$ with the following property: if
\begin{align}
\| e^{i(t-t_0)\Delta}(u(t_0)-w(t_0))\|_{W_\tas(I)}&\leq \beta,\label{condition a}\\
\| |\nabla|^{\frac{4}{d+2}}e^{i(t-t_0)\Delta}(u(t_0)-w(t_0))\|_{X (I)}&\leq \beta,\label{condition aa}\\
\|\la \nabla \ra e\|_{L_{t,x}^{\frac{2(d+2)}{d+4}}(I)}&\leq\beta\label{condition b}
\end{align}
for some $0<\beta<\alpha$, then
\begin{align}
\|\la\nabla \ra u\|_{S(I)}\lesssim_{B_1,B_2,B_3}1.
\end{align}
\end{theorem}
\begin{remark}
By interpolation, \eqref{condition a} and \eqref{condition aa} can be replaced by the stronger condition
\begin{align}
\| \la\nabla\ra e^{i(t-t_0)\Delta}(u(t_0)-w(t_0))\|_{W_\tas(I)}&\leq \beta,
\end{align}
which is the smallness condition given in the long time perturbation theory \cite[Prop. 3.2]{MiaoDoubleCrit}.
\end{remark}
For the precise definition of the function spaces defined in Theorem \ref{long time pert}, we refer to Section \ref{notations and definitions} below for details. The main challenge for proving Theorem \ref{long time pert} lies in the fact that both nonlinearities of \eqref{NLS} are endpoint critical nonlinearities and there is no chance to estimate one by another using interpolation. This will force us to directly derive suitable estimates for both of the nonlinearities using fractional calculus.

As a direct consequence, we immediately deduce the following generalization of Theorem \ref{thm miao}. The proof is a straightforward modification of the arguments from \cite{MiaoDoubleCrit}, thus we omit the details.
\begin{theorem}\label{thm miao luo}
Theorem \ref{thm miao} continues to hold for all $d\geq 5$.
\end{theorem}

The rest of the paper is organized as follows: In Section \ref{notations and definitions} we introduce the notation and definitions which will be used throughout the paper. In Section \ref{pert section} we give the proof of Theorem \ref{long time pert}.

\subsection{Notations and definitions}\label{notations and definitions}
We use the notation $A\lesssim B$ whenever there exists some positive constant $C$ such that $A\leq CB$. Similarly we define $A\gtrsim B$ and we use $A\sim B$ when $A\lesssim B\lesssim A$. For an interval $I\subset \R$, the space $L_t^qL_x^r(I)$ is defined by
\begin{align*}
L_t^qL_x^r(I):=\{u:I\times \R^2\to\C:\|u\|_{L_t^qL_x^r(I)}<\infty\},
\end{align*}
where
\begin{align*}
\|u\|^q_{L_t^qL_x^r(I)}:=\int_{\R}\|u\|^q_r\,dt.
\end{align*}
When $q=r$, we simply write $L_{t,x}^q:=L_t^qL_x^r$. A pair $(q,r)$ is said to be $\dot{H}^s$-admissible with $s\in[0,1]$ if $q,r\in[2,\infty]$ and $\frac{2}{q}+\frac{d}{r}=\frac{d}{2}-s$. When $s=0$, we simply say the pair $(q,r)$ is $L^2$-admissible. For any $L^2$-admissible pairs $(q_1,r_1)$ and $(q_2,r_2)$ we have the following Strichartz estimate: if $u$ is a solution of
\begin{align*}
i\pt_t u+\Delta u=F(u)
\end{align*}
in $I\subset\R$ with $t_0\in I$ and $u(t_0)=u_0$, then
\begin{align*}
\|u\|_{L_t^q L_x^r(I)}\lesssim \|u_0\|_2+\|F(u)\|_{L_t^{q_2'} L_x^{r_2'}(I)},
\end{align*}
where $(q_2',r_2')$ is the H\"older conjugate of $(q_2,r_2)$. For a proof, we refer to \cite{EndpointStrichartz,Cazenave2003}. For $s\in\R$, the multipliers $|\nabla|^s$ and $\la\nabla\ra^s$ are defined by the symbols
\begin{align*}
|\nabla|^s f(x)&=\mathcal{F}^{-1}\bg(|\xi|^s\hat{f}(\xi)\bg)(x),\\
\la\nabla\ra^s f(x)&=\mathcal{F}^{-1}\bg((1+|\xi|^2)^{\frac{s}{2}}\hat{f}(\xi)\bg)(x).
\end{align*}
The following function spaces will be used throughout the paper:
\begin{align*}
W_{\tas}&:=L_{t,x}^{\frac{2(d+2)}{d}},\quad W_{\tbs}:=L_{t,x}^{\frac{2(d+2)}{d-2}},\\
V_{\tbs}&:=L_t^{\frac{2(d+2)}{d-2}}L_x^{\frac{2d(d+2)}{d^2+4}},\\
S&:=L_t^\infty L_x^2\cap L_t^2 L_x^{\tbs},\\
X&:=L_t^{\frac{d(d+2)}{2(d-2)}}L_x^{\frac{2d^2(d+2)}{d^3-4d+16}},\\
Y&:=L_t^{\frac{d}{2}}L_x^{\frac{2d^2(d+2)}{d^3+4d^2+4d-16}},\\
Z&:=L_t^{\frac{d(d+2)}{2(d-2)}}L_{x}^{\frac{2d^2(d+2)}{d^3+2d^2-8d+16}}.
\end{align*}
One easily verifies using H\"older and Sobolev that
\begin{align}
\||\nabla|^{s}\bg(|u|^{\frac{4}{d}}u\bg)\|_{L_{t,x}^{\frac{2(d+2)}{d+4}}(I)}&\lesssim \||\nabla|^s u\|_{W_{\tas}(I)}\|u\|^{\frac{4}{d}}_{W_{\tas}(I)},\label{inter1}\\
\||\nabla|^{s}\bg(|u|^{\frac{4}{d-2}}u\bg)\|_{L_{t,x}^{\frac{2(d+2)}{d+4}}(I)}&\lesssim \||\nabla|^s u\|_{W_{\tas}(I)}\|u\|^{\frac{4}{d-2}}_{W_{\tbs}(I)}\label{inter3}
\end{align}
for $s\in\{0,1\}$, and
\begin{align}
\|u\|_{W_{\tbs}(I)}&\lesssim\|\nabla u\|_{V_\tbs(I)}\leq\|\nabla u\|_{L_t^\infty L_x^2(I)}^{\frac{2}{d}}\|\nabla u\|_{W_{\tas}(I)}^{1-\frac{2}{d}}\label{inter4}.
\end{align}
We also record here the following useful elementary inequalities: By fundamental calculus we have the following elementary inequality
\begin{align}
\bg|\bg|\sum_{j=1}^kz^j\bg|^{\alpha}-\sum_{j=1}^k|z^j|^{\alpha}\bg|&\lesssim\sum_{1\leq i,j\leq k,i\neq j}|z^i|^{\alpha-1}|z^j|\label{tele2}
\end{align}
for $z\in\C^k$ and $\alpha\in(1,\infty)$; For function $H(z)=|z|^{\frac{4}{d}}z$ we have
\begin{align}
|(H(u)-H(v)|\lesssim (|u|^{\frac{4}{d}}+|v|^{\frac{4}{d}})|u-v|\label{diff1}
\end{align}
for $d\geq3$,
\begin{align}
\bg|\nabla \bg(H(u)-H(v)\bg)\bg|\lesssim|u|^{\frac{4}{d}}|\nabla u-\nabla v|+|\nabla v|(|u|^\frac{4-d}{d}+|v|^\frac{4-d}{d})|u-v|\label{diff2}
\end{align}
for $d=3$, and
\begin{align}
\bg|\nabla \bg(H(u+v)-H(u)-H(v)\bg)\bg|\lesssim|u|^{\frac{4}{d}}|\nabla v|+|v|^{\frac{4}{d}}|\nabla u|\label{diff3}
\end{align}
for $d\geq 4$.
\section{Perturbation theory}\label{pert section}
In this section we prove Theorem \ref{long time pert}. The proof relies on a modification of the arguments involving fractional calculus given in \cite{KillipVisanNotes}. To proceed, we first record some auxiliary tools (Lemma \ref{fractional product rule} to Lemma \ref{lemma of exotic 1}). For details of their proofs, we refer to \cite{KillipVisanNotes} and the references therein. We will also restrict the space dimension $d$ to be larger than four in this section.
\begin{lemma}\label{fractional product rule}
Let $s\in(0,1]$ and $q,q_1,q_2,q_3,q_4\in(1,\infty)$ with
\begin{align*}
\frac{1}{q}=\frac{1}{q_1}+\frac{1}{q_2}=\frac{1}{q_3}+\frac{1}{q_4}.
\end{align*}
Then
\begin{align}
\||\nabla|^s(uv)\|_{q}\lesssim\|u\|_{{q_1}}\||\nabla|^s v\|_{{q_2}}+\||\nabla|^s u\|_{{q_3}}\|v\|_{{q_4}}.
\end{align}
\end{lemma}

\begin{lemma}\label{fractional chain rule}
Let $G:\C\to \C$ be a $C^1$-function and let $s\in(0,1]$. Then for all $1<p,p_1,p_2<\infty$ with $\frac{1}{p}=\frac{1}{p_1}+\frac{1}{p_2}$ we have
\begin{align}
\||\nabla|^sG(u)\|_{p}\lesssim \|G'(u)\|_{p_1}\||\nabla|^s u\|_{p_2}.
\end{align}
\end{lemma}

\begin{lemma}\label{fractional holder cont}
Let $h:\C\to\C$ be a H\"older continuous function of order $\alpha\in(0,1)$. Then for any $s\in(0,\alpha)$, $q\in(1,\infty)$ and $\sigma\in(\frac{s}{\alpha},1)$ we have
\begin{align}
\||\nabla|^s h(u)\|_{q}\lesssim\|u\|^{\alpha-\frac{s}{\sigma}}_{{(\alpha-\frac{s}{\sigma})q_1}}
\||\nabla|^\sigma u\|_{{\frac{s}{\sigma}q_2}}^{\frac{s}{\sigma}},
\end{align}
provided that $\frac{1}{q}=\frac{1}{q_1}+\frac{1}{q_2}$ and $(\alpha-\frac{s}{\sigma})q_1>\alpha$.
\end{lemma}


\begin{lemma}\label{lemma of exotic 3}
For any interval $I$ we have
\begin{align}
\||\nabla|^{\frac{4}{d+2}}u\|_{X(I)}&\lesssim \|u\|^{\frac{1}{d+2}}_{W_\tbs(I)}\|\nabla u\|^{\frac{d+1}{d+2}}_{S(I)}\lesssim \|\nabla u\|_{S(I)},\label{inter a}\\
\|u\|_{W_\tbs(I)}&\lesssim \||\nabla|^{\frac{4}{d+2}}u\|^{c}_{X(I)}\|\nabla u\|^{1-c}_{S(I)}\label{inter d}
\end{align}
for some $c=c(d)\in(0,1)$.
\end{lemma}

\begin{lemma}\label{lemma of exotic 0}
For any interval $I$ we have
\begin{align}\label{inter e}
\bg\|\int_{t_0}^t |\nabla|^{\frac{4}{d+2}}e^{i(t-s)\Delta}f(s)\,ds\bg\|_{X(I)}\lesssim \||\nabla|^{\frac{4}{d+2}}f\|_{Y(I)}.
\end{align}
\end{lemma}

\begin{lemma}\label{lemma of exotic 1}
For any interval $I$ we have
\begin{align}\label{inter f}
\||\nabla|^{\frac{4}{d+2}}(|u|^{\frac{4}{d-2}}u)\|_{Y(I)}\lesssim \||\nabla|^{\frac{4}{d+2}}u\|^{\frac{d+2}{d-2}}_{X(I)}
\end{align}
and
\begin{align}\label{inter g}
&\,\||\nabla|^{\frac{4}{d+2}}(|u+w|^{\frac{4}{d-2}}v)\|_{Y(I)}\nonumber\\
\lesssim&\,\bg( \||\nabla|^{\frac{4}{d+2}}u\|^{\frac{8}{d^2-4}}_{X(I)}\|\nabla u\|^{\frac{4d}{d^2-4}}_{S(I)}+\||\nabla|^{\frac{4}{d+2}}w\|^{\frac{8}{d^2-4}}_{X(I)}\|\nabla w\|^{\frac{4d}{d^2-4}}_{S(I)}\bg)\||\nabla|^{\frac{4}{d+2}}v\|_{X(I)}.
\end{align}
\end{lemma}

Next, we prove an exotic dual Strichartz estimate for the mass-critical term:
\begin{lemma}\label{lemma of exotic 2}
For any interval $I$ we have
\begin{align}\label{exotic 1}
\||\nabla|^{\frac{4}{d+2}}(|u|^{\frac{4}{d}}u)\|_{Y(I)}\lesssim \|u\|^{\frac{4(1-c)}{d}}_{S(I)}\|u\|^{\frac{4c}{d}}_{Z(I)}\||\nabla|^{\frac{4}{d+2}}u\|_{X(I)}
\end{align}
and
\begin{align}\label{exotic 2}
&\,\||\nabla|^{\frac{4}{d+2}}(|u+w|^{\frac{4}{d}}v)\|_{Y(I)}\nonumber\\
\lesssim&\,\bg(\|u\|^{\frac{4(1-c)}{d}}_{S(I)}\|u\|^{\frac{4c}{d}}_{Z(I)}
+\|w\|^{\frac{4(1-c)}{d}}_{S(I)}\|w\|^{\frac{4c}{d}}_{Z(I)}\bg)\||\nabla|^{\frac{4}{d+2}}v\|_{X(I)}\nonumber\\
&\,+\bg(\|\la\nabla\ra u\|^{\frac{4}{d+1}}_{S}+\|\la\nabla\ra w\|^{\frac{4}{d+1}}_{S}\bg)\bg(\|u\|^{\frac{4(1-c)}{d(d+1)}}_{S(I)}\|u\|^{\frac{4c}{d(d+1)}}_{Z(I)}+
\|w\|^{\frac{4(1-c)}{d(d+1)}}_{S(I)}\|w\|^{\frac{4c}{d(d+1)}}_{Z(I)}\bg)\nonumber\\
&\,\times \bg(\|v\|^{\frac{4}{d+1}}_{Z(I)}
\||\nabla|^{\frac{4}{d+2}}v\|^{\frac{d-3}{d+1}}_{X(I)}\bg)
\end{align}
for some $c=c(d)\in(0,1)$.
\end{lemma}

\begin{proof}
To simplify notations we omit the symbol $I$ in the following proof. First notice that using interpolation, for any $L^2$-admissible pair $(q,r)$ with $q\in(2,\tbs)$ we have
\begin{align}\label{l2 interpolation}
\|u\|_{L_t^q L_x^r}\lesssim \|u\|^{1-c}_{S}\|u\|^c_{Z}
\end{align}
for some $c=c(q,r)\in(0,1)$. For different $L^2$-admissible pairs $(q_1,r_1)$ and $(q_2,r_2)$ with $c_1=c(q_1,r_1)>c(q_2,r_2)=c_2$, we also have
\begin{align}
 \|u\|^{1-c_1}_{S}\|u\|^{c_1}_{Z}= \|u\|^{1-c_1}_{S}\|u\|^{c_1-c_2}_{Z}\|u\|^{c_2}_{Z}\lesssim  \|u\|^{1-c_2}_{S}\|u\|^{c_2}_{Z}.
\end{align}
We will thus refer to a unified (and possibly small) $c\in(0,1)$ by the application of \eqref{l2 interpolation} for different $L^2$-admissible pairs. Using Lemma \ref{fractional chain rule} we infer that
\begin{align}
&\,\||\nabla|^{\frac{4}{d+2}}(|u|^{\frac{4}{d}}u)\|_{Y}\nonumber\\
\lesssim &\,\|u^{\frac{4}{d}}\|_{L_t^{\frac{d(d+2)}{8}}L_x^{\frac{d^2(d+2)}{2(d^2+2d-8)}}}\||\nabla|^{\frac{4}{d+2}}u\|_{X}\nonumber\\
=&\,\|u\|^{\frac{4}{d}}_{L_t^{\frac{d+2}{2}}L_x^{\frac{2d(d+2)}{d^2+2d-8}}}\||\nabla|^{\frac{4}{d+2}}u\|_{X}\nonumber\\
\lesssim &\,\|u\|^{\frac{4(1-c)}{d}}_{S}\|u\|^{\frac{4c}{d}}_{Z}\||\nabla|^{\frac{4}{d+2}}u\|_{X},
\end{align}
which gives \eqref{exotic 1}. By Lemma \ref{fractional product rule} and \eqref{l2 interpolation} we have
\begin{align}
&\,\||\nabla|^{\frac{4}{d+2}}(|u+w|^{\frac{4}{d}}v)\|_{L_t^{\frac{d}{2}}L_x^{\frac{2d^2(d+2)}{d^3+4d^2+4d-16}}}\nonumber\\
\leq&\,\|(u+w)^{\frac{4}{d}}\|_{L_t^{\frac{d(d+2)}{8}}L_x^{\frac{d^2(d+2)}{2(d-2)(d+4)}}}\||\nabla|^{\frac{4}{d+2}}v\|_{X}\nonumber\\
&\,+\||\nabla|^{\frac{4}{d+2}}(|u+w|^{\frac{4}{d}})\|_{L_t^{\frac{d(d+2)}{8}}L_x^{\frac{d^2(d+1)(d+2)}{2(d^3+3d^2-8d-8)}}}
\|v\|_{L_t^{\frac{d(d+2)}{2(d-2)}}L_{x}^{\frac{2d^2(d+1)(d+2)}{d^4+d^3-4d^2+20d+16}}}\nonumber\\
=&\,\|u+w\|^{\frac{4}{d}}_{L_t^{\frac{(d+2)}{2}}L_x^{\frac{2d(d+2)}{(d-2)(d+4)}}}\||\nabla|^{\frac{4}{d+2}}v\|_{X}\nonumber\\
&\,+\||\nabla|^{\frac{4}{d+2}}(|u+w|^{\frac{4}{d}})\|_{L_t^{\frac{d(d+2)}{8}}L_x^{\frac{d^2(d+1)(d+2)}{2(d^3+3d^2-8d-8)}}}
\|v\|_{L_t^{\frac{d(d+2)}{2(d-2)}}L_{x}^{\frac{2d^2(d+1)(d+2)}{d^4+d^3-4d^2+20d+16}}}\nonumber\\
\lesssim &\,\bg(\|u\|^{\frac{4(1-c)}{d}}_{S}\|u\|^{\frac{4c}{d}}_{Z}
+\|w\|^{\frac{4(1-c)}{d}}_{S}\|w\|^{\frac{4c}{d}}_{Z}\bg)\||\nabla|^{\frac{4}{d+2}}v\|_{X}\nonumber\\
&\,+\||\nabla|^{\frac{4}{d+2}}(|u+w|^{\frac{4}{d}})\|_{L_t^{\frac{d(d+2)}{8}}L_x^{\frac{d^2(d+1)(d+2)}{2(d^3+3d^2-8d-8)}}}
\|v\|_{L_t^{\frac{d(d+2)}{2(d-2)}}L_{x}^{\frac{2d^2(d+1)(d+2)}{d^4+d^3-4d^2+20d+16}}}.\label{exotic 3}
\end{align}
It is left to estimate the second product in \eqref{exotic 3}. Using H\"older and Sobolev we obtain that
\begin{align}
&\,\|v\|_{L_t^{\frac{d(d+2)}{2(d-2)}}L_{x}^{\frac{2d^2(d+1)(d+2)}{d^4+d^3-4d^2+20d+16}}}\nonumber\\
\leq &\,\|v\|^{\frac{4}{d+1}}_{Z}
\|v\|^{\frac{d-3}{d+1}}_{L_t^{\frac{d(d+2)}{2(d-2)}}L_{x}^{\frac{2d^2(d+2)}{(d-2)^2(d+4)}}}\nonumber\\
\leq &\,\|v\|^{\frac{4}{d+1}}_{Z}
\||\nabla|^{\frac{4}{d+2}}v\|^{\frac{d-3}{d+1}}_{X}.\label{exotic 4}
\end{align}
On the other hand, by Lemma \ref{fractional holder cont} we know that
\begin{align}
&\,\||\nabla|^{\frac{4}{d+2}}(|u+w|^{\frac{4}{d}})\|_{L_x^{\frac{d^2(d+1)(d+2)}{2(d^3+3d^2-8d-8)}}}\nonumber\\
\leq &\,\||\nabla|^{\frac{d+1}{d+2}}(u+w)\|^{\frac{4}{d+1}}_{L_x^{\frac{2d(d+2)}{d^2+2d-10}}}
\|u+w\|^{\frac{4}{d(d+1)}}_{L_x^{\frac{2d(d+2)}{d^2+2d-8}}},
\end{align}
where we set
\begin{align*}
\sigma=\frac{d+1}{d+2},\quad q_1=\frac{d^2(d+1)(d+2)}{2(d^2+2d-8)},\quad q_2=\frac{d(d+1)(d+2)}{2(d^2+2d-10)}
\end{align*}
therein. Using H\"older, Sobolev and interpolation we finally conclude that
\begin{align}
&\,\||\nabla|^{\frac{4}{d+2}}(|u+w|^{\frac{4}{d}})\|_{L_t^{\frac{d(d+2)}{8}}L_x^{\frac{d^2(d+1)(d+2)}{2(d^3+3d^2-8d-8)}}}\nonumber\\
\leq&\,\||\nabla|^{\frac{d+1}{d+2}}(u+w)\|^{\frac{4}{d+1}}_{L_t^{\frac{d+2}{2}}L_x^{\frac{2d(d+2)}{d^2+2d-10}}}
\|u+w\|^{\frac{4}{d(d+1)}}_{L_t^{\frac{d+2}{2}}L_x^{\frac{2d(d+2)}{d^2+2d-8}}}\nonumber\\
\leq&\,\bg(\|\la\nabla\ra u\|^{\frac{4}{d+1}}_{S}+\|\la\nabla\ra w\|^{\frac{4}{d+1}}_{S}\bg)\nonumber\\
&\,\times\bg(\|u\|^{\frac{4(1-c)}{d(d+1)}}_{S}\|u\|^{\frac{4c}{d(d+1)}}_{Z}
+\|w\|^{\frac{4(1-c)}{d(d+1)}}_{S}\|w\|^{\frac{4c}{d(d+1)}}_{Z}\bg)\label{exotic 5}
\end{align}
for some $c\in(0,1)$. \eqref{exotic 3}, \eqref{exotic 4} and \eqref{exotic 5} then imply \eqref{exotic 2}.
\end{proof}

Next, we formulate a small data well-posedness result for \eqref{NLS}, which is slightly different from the standard one and is suitable for the proof of Lemma \ref{short time pert} given below.
\begin{lemma}\label{well posedness lemma}
For any $A>0$ there exists some $\beta>0$ such that the following is true: Suppose that $t_0\in I$ for some interval $I$. Suppose also that $u_0=u(t_0)\in H^1(\R^d)$ with
\begin{align*}
\|u_0\|_{H^1}\leq A
\end{align*}
and
\begin{align*}
\| e^{i(t-t_0)\Delta}u_0\|_{W_{\tas}(I)}+\||\nabla|^{\frac{4}{d+2}} e^{i(t-t_0)\Delta}u_0\|_{X(I)}\leq \beta.
\end{align*}
Then \eqref{NLS} has a unique solution $u\in C(I;H^1(\R^d))$ such that
\begin{align*}
\|\la\nabla\ra u\|_{S(I)}&\lesssim \|u_0\|_{H^1},\\
\|u\|_{W_{\tas}(I)}&\leq 2\| e^{i(t-t_0)\Delta}u_0\|_{W_{\tas}(I)},\\
\||\nabla|^{\frac{4}{d+2}} u\|_{X(I)}&\leq 2\| |\nabla|^{\frac{4}{d+2}}e^{i(t-t_0)\Delta}u_0\|_{X(I)}
\end{align*}
\end{lemma}
\begin{proof}
We define the space $B(I)$ by
\begin{align}
B(I):=\bg\{u\in L_t^\infty H_x^1(I):\|\la\nabla\ra u\|_{S(I)}&\leq 2C \|u_0\|_{H^1},\nonumber\\
\|u\|_{W_{\tas}(I)}&\leq 2\| e^{it\Delta}u_0\|_{W_{\tas}(I)}\nonumber\\
\||\nabla|^{\frac{4}{d+2}} u\|_{X(I)}
&\leq 2\||\nabla|^{\frac{4}{d+2}} e^{it\Delta}u_0\|_{X(I)}
\bg\}.
\end{align}
One easily checks that the set $B(I)$ equipped with the metric $\rho$ defined by
$$\rho(u,v):=\|u-v\|_{S(I)}$$
is a complete metric space. Now define the operator $\Phi$ by
\begin{align}
\Phi(u):= e^{i(t-t_0)\Delta}u_0-i\int_{t_0}^t e^{i(t-s)\Delta}(|u|^{\frac{4}{d}}u-|u|^{\frac{4}{d-2}}u)\,ds.
\end{align}
We show that $\Phi$ is a contraction on $B(I)$. Using \eqref{inter d} we infer that there exists some $c\in(0,1)$ such that
\begin{align}
\|u\|_{W_\tbs(I)}\leq (2CA)^{1-c}\||\nabla|^{\frac{4}{d+2}}u\|^c_{X(I)}\leq (2CA)^{1-c}2^{c}\beta^c.\label{new 3}
\end{align}
Combining with Strichartz, \eqref{inter1} and \eqref{inter3} we deduce that
\begin{align}
&\,\|\la\nabla\ra \Phi(u)\|_{S(I)}\nonumber\\
\leq &\,\|\la\nabla\ra e^{it\Delta}u_0\|_{S(I)}+C\sum_{s=1}^2\bg(\||\nabla|^s\bg(|u|^{\frac{4}{d}} u\bg)\|_{L_{t,x}^{\frac{2(d+2)}{d+4}}(I)}\bg)\nonumber\\
&\,+C\sum_{s=1}^2\bg(\||\nabla|^s\bg(|u|^{\frac{4}{d-2}} u\bg)\|_{L_{t,x}^{\frac{2(d+2)}{d+4}}(I)}\bg)\nonumber\\
\leq&\, \|\la\nabla\ra e^{it\Delta}u_0\|_{S(I)}+C\|\la\nabla\ra u\|_{W_{\tas}(I)}\|u\|^{\frac{4}{d}}_{W_{\tas}(I)}+C\|\la\nabla\ra u\|_{W_{\tas}(I)}\|u\|^{\frac{4}{d-2}}_{W_{\tbs}(I)}\nonumber\\
\leq&\, \|\la\nabla\ra e^{it\Delta}u_0\|_{S(I)}+C\|\la\nabla\ra u\|_{S(I)}\|u\|^{\frac{4}{d}}_{W_{\tas}(I)}+C\|\la\nabla\ra u\|_{S(I)}\|u\|^{\frac{4}{d-2}}_{W_{\tbs}(I)}\nonumber\\
\leq& \,C\| u_0\|_{H^1}+2C((2\beta)^{\frac{4}{d}}+((2CA)^{1-c}2^{c}\beta^c)^{\frac{4}{d-2}})\|u_0\|_{H^1}.\label{modi1}
\end{align}
Analogously, we have
\begin{align}
&\,\|\Phi(u)\|_{W_{\tas}(I)}\nonumber\\
\leq&\, \| e^{it\Delta}u_0\|_{W_\tas(I)}+C\||u|^{\frac{4}{d}} u\|_{L_{t,x}^{\frac{2(d+2)}{d+4}}(I)}
+C\||u|^{\frac{4}{d-2}} u\|_{L_{t,x}^{\frac{2(d+2)}{d+4}}(I)}\nonumber\\
\leq&\, \|e^{it\Delta}u_0\|_{W_\tas(I)}+C\|u\|^{1+\frac{4}{d}}_{W_{\tas}(I)}
+C\| u\|_{W_{\tas}(I)}\|u\|^{\frac{4}{d-2}}_{W_{\tbs}(I)}\nonumber\\
\leq& \, \|e^{it\Delta}u_0\|_{W_\tas(I)}+2C((2\beta)^{\frac{4}{d}}+((2CA)^{1-c}2^{c}\beta^c)^{\frac{4}{d-2}})\|e^{it\Delta}u_0\|_{W_\tas(I)}.\label{modi2}
\end{align}
Using \eqref{inter f} and \eqref{exotic 1} we see that
\begin{align}
&\,\||\nabla|^{\frac{4}{d+2}}\Phi(u)\|_{X(I)}\nonumber\\
\leq&\,\| |\nabla|^{\frac{4}{d+2}}e^{it\Delta}u_0\|_{X(I)}+C\||\nabla|^{\frac{4}{d+2}}\bg(|u|^{\frac{4}{d}} u\bg)\|_{Y(I)}
+C\||\nabla|^{\frac{4}{d+2}}\bg(|u|^{\frac{4}{d-2}} u\bg)\|_{Y(I)}\nonumber\\
\leq&\,\| |\nabla|^{\frac{4}{d+2}}e^{it\Delta}u_0\|_{X(I)}+C\||\nabla|^{\frac{4}{d+2}}u\|^{\frac{d+2}{d-2}}_{X(I)}
+C(2CA)^{\frac{4(1-c)}{d}}\|u\|^{\frac{4c}{d}}_{Z(I)}\||\nabla|^{\frac{4}{d+2}}u\|_{X(I)}.\label{modi3}
\end{align}
Since $Z$ corresponds to an $L^2$-admissible pair, we know that there exists some $\kappa\in(0,1)$ such that
\begin{align}
\|u\|_{Z(I)}\leq \|u\|^{1-\kappa}_{S(I)}\|u\|^\kappa_{W_\tas(I)}.
\end{align}
Summing up we have
\begin{align}
&\,\||\nabla|^{\frac{4}{d+2}}\Phi(u)\|_{X(I)}\nonumber\\
\leq&\,\| |\nabla|^{\frac{4}{d+2}}e^{it\Delta}u_0\|_{X(I)}+2C(2\beta)^{\frac{4}{d-2}}\| |\nabla|^{\frac{4}{d+2}}e^{it\Delta}u_0\|_{X(I)}\nonumber\\
&\,+2C(2CA)^{\frac{4(1-c)}{d}}(2CA)^{\frac{4c(1-\kappa)}{d}}(2\beta)^{\frac{4c\kappa}{d}}
\||\nabla|^{\frac{4}{d+2}}e^{it\Delta}u_0\|_{X(I)},
\end{align}
Hence by choosing ${\beta}$ sufficiently small we see that $\Phi$ maps $B(I)$ into itself. In a similar way, using \eqref{diff1} followed by Strichartz, H\"older and \eqref{new 3} we obtain that
\begin{align}
&\,\|\Phi(u)-\Phi(v)\|_{S(I)}\nonumber\\
\leq &\,C(\|u\|^{\frac{4}{d}}_{W_\tas(I)}+\|v\|^{\frac{4}{d}}_{W_\tas(I)}+\|u\|^{\frac{4}{d-2}}_{W_\tbs(I)}+\|v\|^{\frac{4}{d-2}}_{W_\tbs(I)})\|u-v\|_{W_\tas(I)}\nonumber\\
\leq&\,C\bg(\|u\|^{\frac{4}{d}}_{W_\tas(I)}+\|v\|^{\frac{4}{d}}_{W_\tas(I)}\nonumber\\
&\,+(2CA)^{\frac{4(1-c)}{d-2}}\||\nabla|^{\frac{4}{d+2}}u\|^{\frac{4c}{d-2}}_{X(I)}+(2CA)^{\frac{4(1-c)}{d-2}}\||\nabla|^{\frac{4}{d+2}}v\|^{\frac{4c}{d-2}}_{X(I)}\bg)
\|u-v\|_{S(I)}\nonumber\\
\leq&\,2C\big((2{\beta})^{\frac{4}{d}}+(2CA)^{\frac{4(1-c)}{d-2}}(2\beta)^{\frac{4c}{d-2}}\big)\|u-v\|_{S(I)}.
\end{align}
Thus choosing even smaller ${\beta}$ if necessary we infer that $\Phi$ is a contraction on $B(I)$. Now the existence and uniqueness of a solution $u$ of \eqref{NLS} are ensured by the Banach fixed point theorem. The continuity of $u$ follows immediately from the fact that $u$ satisfies the integral equation.
\end{proof}

\begin{lemma}[Short time perturbation]\label{short time pert}
Let $u\in C(I;H^1(\R^d))$ be a solution of \eqref{NLS} defined on some interval $I\ni t_0$. Assume also that $w$ is an approximate solution of the following perturbed NLS
\begin{align}
i\pt_t w+\Delta w=|w|^{\frac{4}{d}}w-|w|^{\frac{4}{d-2}}w+e
\end{align}
such that
\begin{align}
\|w\|_{L_t^\infty H_x^1(I)}\leq B_1,\quad \|u(t_0)-w(t_0)\|_{H^1}\leq B_2
\end{align}
for some $B_1,B_2>0$. Then there exist some positive $\beta_0,\beta_1\ll  1$, depending on $B_1$ and $B_2$, with the following property: if
\begin{align}\label{condition 1}
\||\nabla|^{\frac{4}{d+2}}w\|_{X(I)}+\|\la\nabla\ra w\|_{W_\tas(I)}\leq\beta_0
\end{align}
and
\begin{align}
\| e^{i(t-t_0)\Delta}(u(t_0)-w(t_0))\|_{W_\tas(I)}&\leq \beta,\label{condition 3}\\
\||\nabla|^{\frac{4}{d+2}} e^{i(t-t_0)\Delta}(u(t_0)-w(t_0))\|_{X(I)}&\leq \beta,\label{condition 2}\\
\|\la \nabla \ra e\|_{L_{t,x}^{\frac{2(d+2)}{d+4}}(I)}&\leq\beta\label{condition 4}
\end{align}
for some $0<\beta\leq \beta_1$, then
\begin{align}
\| u-w\|_{Z\cap W_\tas (I)}&\lesssim \beta^\kappa,\label{short time 2}\\
\||\nabla|^{\frac{4}{d+2}}(u-w)\|_{X(I)}&\lesssim \beta^{\kappa},\label{short time 1}\\
\|F(u)-F(w)\|_{L_{t,x}^{\frac{2(d+2)}{d+4}}(I)}&\lesssim \beta^\kappa,\label{short time 3}\\
\||\nabla|^{\frac{4}{d+2}}(F(u)-F(w))\|_{Y(I)}&\lesssim \beta^\kappa,\label{short time 6}\\
\|\la \nabla\ra u\|_{S(I)}+\|\la \nabla\ra w\|_{S(I)}&\lesssim B_1+B_2\label{short time 5}
\end{align}
for some $\kappa\in(0,1)$. Here $F(z)=-|z|^{\frac{4}{d}}z+|z|^{\frac{4}{d-2}}z$.
\end{lemma}

\begin{proof}
First denote by $\kappa_1$ the number such that
\begin{align}\label{kappa}
\|u\|_{Z}&\lesssim \|u\|^{1-\kappa_1}_S\|u\|_{W_\tas}^{\kappa_1}.
\end{align}
By definition of $u$ and $w$ we have
\begin{align}
u(t)&=e^{i(t-t_0)\Delta}u(t_0)+i\int_{t_0}^t e^{i(t-s)\Delta} F(u)\,ds,\\
w(t)&=e^{i(t-t_0)\Delta}w(t_0)+i\int_{t_0}^t e^{i(t-s)\Delta} (F(w)-e)\,ds.
\end{align}
Using Strichartz, \eqref{inter1}, \eqref{inter3}, \eqref{condition 1} and \eqref{inter4} we obtain that
\begin{align}
&\,\|\la\nabla \ra w\|_{S(I)}\nonumber\\
\lesssim&\,\|w(t_0)\|_{H_x^1}+\bg(\|w\|^{\frac{4}{d}}_{W_\tas(I)}+\|w\|^{\frac{4}{d-2}}_{W_\tbs(I)}\bg)\|\la\nabla \ra w\|_{S(I)}+\|\la \nabla \ra e\|_{L_{t,x}^{\frac{2(d+2)}{d+4}}}\nonumber\\
\lesssim& \,B_1+\beta+(\beta_0^{\frac{4}{d}}+B_1^{\frac{8}{d(d-2)}}\beta_0^{\frac{4}{d}})\|\la\nabla \ra w\|_{S(I)}.
\end{align}
By choosing $\beta_0$ sufficiently small we infer that
\begin{align}\label{inter aa}
\|\la\nabla \ra w\|_{S(I)}\lesssim B_1.
\end{align}
Now \eqref{kappa} and \eqref{condition 1} also yield
\begin{align}\label{kappa2}
\|\la\nabla\ra w\|_{Z(I)}\lesssim\beta_0^{\kappa_1}.
\end{align}
Using Strichartz, \eqref{condition 1}, \eqref{inter1}, \eqref{inter3}, \eqref{inter aa}, \eqref{inter4} and \eqref{condition 4} we infer that
\begin{align}
\| e^{i(t-t_0)\Delta}w_0\|_{W_\tas(I)}&\lesssim \| w\|_{W_\tas(I)}+\| F(w)\|_{L_{t,x}^{\frac{2(d+2)}{d+4}}}
+\|e\|_{L_{t,x}^{\frac{2(d+2)}{d+4}}}\nonumber\\
&\lesssim \beta_0+\beta_0^{\frac{4}{d}}+\beta\lesssim \beta_0^{\frac{4}{d}}.
\end{align}
Similarly, \eqref{inter f}, \eqref{exotic 1}, \eqref{inter aa}, \eqref{condition 1} and \eqref{condition 4} yield
\begin{align}
&\,\||\nabla|^{\frac{4}{d+2}}e^{i(t-t_0)\Delta}w_0\|_{X(I)}\nonumber\\
\lesssim& \,\||\nabla|^{\frac{4}{d+2}}w\|_{X(I)}+\||\nabla|^{\frac{4}{d+2}}F(w)\|_{Y(I)}+\|\la\nabla\ra e\|_{L_{t,x}^{\frac{2(d+2)}{d+4}}}\nonumber\\
\lesssim& \,\beta_0+\beta_0^{\frac{d+2}{d-2}}+\beta_0+\beta\lesssim\beta_0.
\end{align}
Combining with \eqref{condition 3}, \eqref{condition 2} and the triangular inequality we deduce that
\begin{align}
\| e^{i(t-t_0)\Delta}u(t_0)\|_{W_\tas(I)}&\lesssim \beta_0^{\frac{4}{d}},\label{interbb}\\
\||\nabla|^{\frac{4}{d+2}}e^{i(t-t_0)\Delta}u(t_0)\|_{X(I)}&\lesssim\beta_0\label{intercc}.
\end{align}
Hence, by choosing $\beta_0$ sufficiently small, we know from Lemma \ref{well posedness lemma} that
\begin{align}
\|\la \nabla\ra u\|_{S(I)}&\lesssim \|u(t_0)\|_{H^1}\lesssim B_1+B_2,\label{bound of nabla u0}\\
\|u\|_{W_{\tas}(I)}&\leq 2\| e^{i(t-t_0)\Delta}u_0\|_{W_{\tas}(I)}\lesssim \beta_0^{\frac{4}{d}},\label{bound of nabla u1}\\
\||\nabla|^{\frac{4}{d+2}} u\|_{X(I)}&\leq 2\| |\nabla|^{\frac{4}{d+2}}e^{i(t-t_0)\Delta}u_0\|_{X(I)}\lesssim \beta_0.\label{bound of nabla u2}
\end{align}
Combining with \eqref{inter aa} we already have \eqref{short time 5}. Using Strichartz, \eqref{interbb}, \eqref{bound of nabla u0}, \eqref{inter1}, \eqref{inter3}, \eqref{bound of nabla u2} and \eqref{inter d} we obtain
\begin{align}
\|\la\nabla\ra u\|_{W_\tas(I)}&\lesssim \| e^{i(t-t_0)\Delta}u(t_0)\|_{W_\tas(I)}+\| F(u)\|_{L_{t,x}^{\frac{2(d+2)}{d+4}}}\nonumber\\
&\lesssim \beta_0^{\frac{4}{d}}+\|\la\nabla\ra u\|_{W_\tas(I)}^{1+\frac{4}{d}}+\beta_0^{\frac{4c}{d-2}}\|\la\nabla\ra u\|_{W_\tas(I)}.
\end{align}
By standard continuity arguments we conclude that
\begin{align}\label{new 1}
\|\la\nabla\ra u\|_{W_\tas(I)}\lesssim \beta_0^{\frac{4}{d}}.
\end{align}
Similarly, from Strichartz, \eqref{intercc}, \eqref{bound of nabla u0}, \eqref{inter f}, \eqref{exotic 1}, \eqref{kappa} and \eqref{new 1} we infer that
\begin{align}
\||\nabla|^{\frac{4}{d+2}}u\|_{X(I)}&\lesssim \||\nabla|^{\frac{4}{d+2}}e^{i(t-t_0)\Delta}u(t_0)\|_{X(I)}+\||\nabla|^{\frac{4}{d+2}}F(u)\|_{Y(I)}\nonumber\\
&\lesssim \beta_0+\||\nabla|^{\frac{4}{d+2}}u\|_{X(I)}^{\frac{d+2}{d-2}}+\beta_0^{\frac{4\kappa_1 c}{d}}\||\nabla|^{\frac{4}{d+2}}u\|_{X(I)}
\end{align}
and using standard continuity arguments we see that
\begin{align}
\||\nabla|^{\frac{4}{d+2}}u\|_{X(I)}\lesssim \beta_0.
\end{align}
Next, we define $v:=u-w$. Then
\begin{align}
v(t)=e^{i(t-t_0)\Delta}(u(t_0)-w(t_0))+i\int_{t_0}^t e^{i(t-s)}(F(v+w)-F(w)+e)\,ds
\end{align}
and
\begin{align}\label{shorttime1}
\|\la\nabla\ra v\|_{S(I)}\leq \|\la\nabla\ra u\|_{S(I)}+\|\la\nabla\ra w\|_{S(I)}\lesssim B_1+B_2.
\end{align}
From Strichartz, \eqref{kappa}, \eqref{bound of nabla u0}, \eqref{inter aa} and \eqref{condition 3} we know that
\begin{align}\label{kappa1}
\| e^{i(t-t_0)\Delta}(u(t_0)-w(t_0))\|_{Z(I)}\lesssim \beta^{\kappa_1}.
\end{align}
By Strichartz, Sobolev, \eqref{condition 3}, \eqref{inter1}, \eqref{inter3}, \eqref{diff1}, \eqref{kappa1}, \eqref{condition 1}, \eqref{inter4}, \eqref{inter d} and \eqref{inter aa} we have
\begin{align}
&\,\| v\|_{Z\cap W_\tas(I)}\nonumber\\
\lesssim& \,\|  e^{i(t-t_0)}(u(t_0)-w(t_0))\|_{Z\cap W_\tas(I)}
+\|  e\|_{L_{t,x}^{\frac{2(d+2)}{d+4}}}\nonumber\\
&\,+\| (F(v+w)-F(w))\|_{L_{t,x}^{\frac{2(d+2)}{d+4}}(I)}\nonumber\\
\lesssim&\, \beta^{\kappa_1}+(\|v\|^{\frac{4}{d}}_{W_{\tas}(I)}+\|v\|^{\frac{4}{d-2}}_{W_{\tbs}(I)}
+\|w\|^{\frac{4}{d}}_{W_{\tas}(I)}+\|w\|^{\frac{4}{d-2}}_{W_{\tbs}(I)})\| v\|_{W_\tas (I)}\nonumber\\
\lesssim&\, \beta^{\kappa_1}+\| v\|^{1+\frac{4}{d}}_{Z\cap W_{\tas}(I)}+\||\nabla|^{\frac{4}{d+2}}v\|^{\frac{4c}{d-2}}_{X(I)}\| v\|_{Z\cap W_{\tas}(I)}
+\beta_0^{\frac{4}{d}}\|v\|_{Z\cap W_\tas (I)}.\label{long 3}
\end{align}
Now using Strichartz, Sobolev, \eqref{kappa2}, \eqref{condition 1}, \eqref{condition 2}, \eqref{shorttime1}, \eqref{inter e}, \eqref{inter g}, \eqref{exotic 2} and the identity
\begin{align*}
F(v+w)-F(w)=v\int_0^1 F_z(v+(1+\theta) w)\,d\theta+\bar{v}\int_0^1 F_{\bar{z}}(v+(1+\theta) w)\,d\theta
\end{align*}
we obtain that
\begin{align}\label{long2}
&\,\||\nabla|^{\frac{4}{d+2}}v\|_{X(I)}\nonumber\\
\lesssim &\,\| |\nabla|^{\frac{4}{d+2}}e^{i(t-t_0)}(u(t_0)-w(t_0))\|_{X(I)}
+\|\la\nabla\ra e\|_{L_{t,x}^{\frac{2(d+2)}{d+4}}}\nonumber\\
&\,+\||\nabla|^{\frac{4}{d+2}}(F(v+w)-F(w))\|_{Y(I)}\nonumber\\
\lesssim &\,\beta+\bg( \||\nabla|^{\frac{4}{d+2}}v\|^{\frac{8}{d^2-4}}_{X(I)}\|\nabla v\|^{\frac{4d}{d^2-4}}_{S(I)}+\||\nabla|^{\frac{4}{d+2}}w\|^{\frac{8}{d^2-4}}_{X(I)}\|\nabla w\|^{\frac{4d}{d^2-4}}_{S(I)}\bg)\||\nabla|^{\frac{4}{d+2}}v\|_{X(I)}\nonumber\\
&\,+\bg(\|v\|^{\frac{4(1-c)}{d}}_{S(I)}\|v\|^{\frac{4c}{d}}_{Z(I)}
+\|w\|^{\frac{4(1-c)}{d}}_{S(I)}\|w\|^{\frac{4c}{d}}_{Z(I)}\bg)\||\nabla|^{\frac{4}{d+2}}v\|_{X(I)}\nonumber\\
&\,+\bg(\|\la\nabla\ra v\|^{\frac{4}{d+1}}_{S(I)}+\|\la\nabla\ra w\|^{\frac{4}{d+1}}_{S(I)}\bg)\bg(\|v\|^{\frac{4(1-c)}{d(d+1)}}_{S(I)}\|v\|^{\frac{4c}{d(d+1)}}_{Z(I)}
+\|w\|^{\frac{4(1-c)}{d(d+1)}}_{S(I)}\|w\|^{\frac{4c}{d(d+1)}}_{Z(I)}\bg)\nonumber\\
&\,\times\bg(\|v\|^{\frac{4}{d+1}}_{Z(I)}
\||\nabla|^{\frac{4}{d+2}}v\|^{\frac{d-3}{d+1}}_{X(I)}\bg)\nonumber\\
\lesssim &\,\beta+\||\nabla|^{\frac{4}{d+2}}v\|^{1+\frac{8}{d^2-4}}_{X(I)}+\beta_0^{\frac{8}{d^2-4}}\||\nabla|^{\frac{4}{d+2}}v\|_{X(I)}
+\|v\|^{\frac{4c}{d}}_{Z(I)}\||\nabla|^{\frac{4}{d+2}}v\|_{X(I)}\nonumber\\
&\,+\beta_0^{\frac{4c\kappa_1}{d}}\||\nabla|^{\frac{4}{d+2}}v\|_{X(I)}
+\|v\|^{\frac{4}{d+1}+\frac{4c}{d(d+1)}}_{Z(I)}
\||\nabla|^{\frac{4}{d+2}}v\|^{\frac{d-3}{d+1}}_{X(I)}\nonumber\\
&\,+\beta_0^{\frac{4c\kappa_1}{d(d+1)}}\|v\|^{\frac{4}{d+1}}_{Z(I)}
\||\nabla|^{\frac{4}{d+2}}v\|^{\frac{d-3}{d+1}}_{X(I)}.
\end{align}
Define
\begin{align}
\|v\|_{W(I)}:=\| v\|_{Z\cap W_\tas(I)}+\||\nabla|^{\frac{4}{d+2}}v\|_{X(I)}.
\end{align}
Adding \eqref{long2} to \eqref{long 3} and absorbing the terms with powers of $\beta_0$ on the r.h.s. to the l.h.s., we obtain
\begin{align}
\|v\|_{W(I)}\lesssim \beta^{\kappa_1}+\|v\|^{1+\frac{4}{d}}_{W(I)}+\|v\|^{1+\frac{4c}{d-2}}_{W(I)}+\|v\|^{1+\frac{8}{d^2-4}}_{W(I)}
+\|v\|^{1+\frac{4c}{d}}_{W(I)}+\|v\|^{1+\frac{4}{d+1}}_{W(I)}.
\end{align}
By standard continuity arguments we infer that
\begin{align}\label{small final1}
\|v\|_{W(I)}\lesssim \beta^{\kappa_1}
\end{align}
and we conclude \eqref{short time 2} and \eqref{short time 1}. \eqref{short time 3} and \eqref{short time 6} follow already from the calculation given in \eqref{long 3} and \eqref{long2}. This completes the proof.
\end{proof}

Having all the preliminaries, we are at the position to prove Theorem \ref{long time pert}.

\begin{proof}[Proof of Theorem \ref{long time pert}]
We first show that
\begin{align}\label{w finiteness S}
\|\la\nabla \ra w\|_{S(I)}\leq C
\end{align}
for some $C=C(B_1,B_2,B_3)>0$. By \eqref{condition c2} we may subdivide $I$ into $J_1=J_1(B_3)$ subintervals $K_j=[s_j,s_{j+1}]$, $j=0,\cdots,J_1-1$, such that
\begin{align}\label{w finiteness W}
\|w\|_{W_\tas\cap W_\tbs(K_j)}\leq \delta_0
\end{align}
for some small $\delta_0$ to be chosen later. On $K_0$, using Strichartz, H\"older and \eqref{condition b} we obtain that
\begin{align}
\|\la\nabla\ra w\|_{S(K_0)}
&\lesssim \|w(s_0)\|_{H^1}+\|\la \nabla \ra e\|_{L_{t,x}^{\frac{2(d+2)}{d+4}}(K_0)}+\|\la \nabla \ra F(w)\|_{L_{t,x}^{\frac{2(d+2)}{d+4}}(K_0)}\nonumber\\
&\lesssim B_1+B_2+\beta+\bg(\|w\|^{\frac{4}{d}}_{W_\tas(K_0)}+\|w\|^{\frac{4}{d-2}}_{W_\tbs(K_0)}\bg)\|\la\nabla \ra w\|_{S(K_0)}\nonumber\\
&\lesssim B_1+B_2+\beta+(\delta_0^{\frac{4}{d}}+\delta_0^{\frac{4}{d-2}})\|\la\nabla \ra w\|_{S(K_0)}.
\end{align}
We can choose $\delta_0$ sufficiently small to absorb the term $(\delta_0^{\frac{4}{d}}+\delta_0^{\frac{4}{d-2}})\|\la\nabla \ra w\|_{S(K_0)}$ to the l.h.s. This yields
\begin{align*}
\|\la\nabla\ra w\|_{S(K_0)}\lesssim B_1+B_2.
\end{align*}
In particular,
\begin{align*}
\|w(s_1)\|_{H^1}\lesssim B_1+B_2.
\end{align*}
Notice also that $\delta_0$ is only dependent on $B_3$. Thus we may iterate the previous step over all $j$ to infer that
\begin{align*}
\|\la\nabla\ra w\|_{S(K_j)}\lesssim B_1+B_2
\end{align*}
for all $j$. Summing all the estimates on $K_j$ over $j$ up yields \eqref{w finiteness S}. Using \eqref{inter a} and \eqref{w finiteness S} we are able to divide $I$ into $J_2=J_2(B_1,B_2,B_3)$ intervals $L_j=[t_j,t_{j+1}]$, $j=0,\cdots,J_{2}-1$, such that
\begin{align}\label{final long 1}
\||\nabla|^{\frac{4}{d+2}}w\|_{X(L_j)}+\|\la\nabla\ra w\|_{W_\tas(L_j)}\leq\beta_0,
\end{align}
with $\beta_0=\beta_0(C(B_1,B_2,B_3),C(B_1,B_2,B_3)+B_1)$ defined by Lemma \ref{short time pert}. By \eqref{condition a} and \eqref{condition aa} we have
\begin{align}
\| e^{i(t-t_0)\Delta}(u(t_0)-w(t_0))\|_{W_\tas(L_0)}&\leq\beta,\label{gg1}\\
\||\nabla|^{\frac{4}{d+2}} e^{i(t-t_0)\Delta}(u(t_0)-w(t_0))\|_{X(L_0)}&\leq\beta\label{gg2}
\end{align}
by setting initially $\alpha=\beta_1$ with $\beta_1=\beta_1(B_1,B_2,B_3)$ from Lemma \ref{short time pert} . Thus Lemma \ref{short time pert} is applicable for $L_0$. In particular, we have for all $j=1,\cdots,J_2-1$ and $\beta\in(0,\alpha)$
\begin{align*}
\|u-w\|_{Z\cap W_\tas (L_j)}&\leq C_0\beta^\kappa,\\
\||\nabla|^{\frac{4}{d+2}}(u-w)\|_{X(L_j)}&\leq C_0\beta^{\kappa},\\
\|F(u)-F(w)\|_{L_{t,x}^{\frac{2(d+2)}{d+4}}(L_j)}&\leq C_0\beta^\kappa,\\
\||\nabla|^{\frac{4}{d+2}}(F(u)-F(w))\|_{Y(L_j)}&\leq C_0\beta^\kappa,\\
\|\la \nabla\ra u\|_{S(L_j)}&\leq C_0C(B_1,B_2,B_3),
\end{align*}
with $\kappa\in(0,1)$ and $C_0=C_0(B_1,B_2,B_3)>0$, provided that
\begin{align}
\| e^{i(t-t_j)\Delta}(u(t_j)-w(t_j))\|_{W_\tas(L_j)}&\leq \beta,\\
\||\nabla|^{\frac{4}{d+2}} e^{i(t-t_j)\Delta}(u(t_j)-w(t_j))\|_{X(L_j)}&\leq \beta
\end{align}
hold for all $j=1,\cdots,J_2-1$. We prove this using inductive arguments. One checks that
\begin{align}
&\,\| e^{i(t-t_{j})\Delta}(u(t_{j})-w(t_{j}))\|_{W_\tas(I_{j})}\nonumber\\
\lesssim &\,\| e^{i(t-t_{0})\Delta}(u(t_{0})-w(t_{0}))\|_{W_\tas(I_{j})}+\|e\|_{L_{t,x}^{\frac{2(d+2)}{d+4}}[t_0,t_{j}]}\nonumber\\
&\,+\|F(u)-F(w)\|_{L_{t,x}^{\frac{2(d+2)}{d+4}}[t_0,t_{j}]}\nonumber\\
\lesssim &\, \beta+\beta+C_0j\beta^\kappa,\\
\nonumber\\
&\,\||\nabla|^{\frac{4}{d+2}}e^{i(t-t_{j})\Delta}(u(t_{j})-w(t_{j}))\|_{X(I_{j})}\nonumber\\
\lesssim &\,\||\nabla|^{\frac{4}{d+2}}e^{i(t-t_{0})\Delta}(u(t_{0})-w(t_{0}))\|_{X(I_{j})}+\|\nabla e\|_{L_{t,x}^{\frac{2(d+2)}{d+4}}[t_0,t_{j}]}\nonumber\\
&\,+\||\nabla|^{\frac{4}{d+2}}(F(u)-F(w))\|_{Y[t_0,t_{j}]}\nonumber\\
\lesssim &\, \beta+\beta+C_0 j\beta^\kappa.
\end{align}
Choosing $\alpha$ iteratively small completes the proof.
\end{proof}

\subsubsection*{Acknowledgments}
The author acknowledges the funding by Deutsche Forschungsgemeinschaft (DFG) through the Priority Programme SPP-1886 (No. NE 21382-1).

\selectlanguage{English}

\addcontentsline{toc}{section}{References}

\end{document}